\documentclass[a4paper,12pt,reqno]{amsart}
\usepackage{amssymb,amsmath,hyperref}
\title[Identities for the eighth order mock theta function $V_{0}(q)$ revisited]{Short proofs of Ramanujan-like identities for the eighth order mock theta function $V_{0}(q)$}

\author{Eric T. Mortenson}
\address{Department of Mathematics and Computer Science, Saint Petersburg State University, Saint Petersburg, 199034, Russia}
\email{etmortenson@gmail.com}

\renewcommand\theta{\vartheta}

\newtheorem{theorem}{Theorem}
\newtheorem{lemma}[theorem]{Lemma}

\newtheorem{proposition}[theorem]{Proposition}
\theoremstyle{definition}

\numberwithin{theorem}{section} 
\numberwithin{equation}{section}

\makeatletter
\@namedef{subjclassname@2020}{\textup{2020} Mathematics Subject Classification}
\makeatother

\begin{document}

\date{29 August 2023}



\begin{abstract}
    Using Appell function properties we give short proofs of Ramanujan-like identities for the eighth order mock theta function $V_0(q)$ after work of Chan and Mao; Mao; and Brietzke, da Silva, and Sellars.   We also present a generalization of the identities in the spirit of celebrated results of Bringmann, Ono, and Rhoades on Dyson's ranks and Maass forms.
\end{abstract}

\maketitle

\section{Introduction}

 Let $q:=e^{2 \pi i \tau}$, $\tau\in\mathbb{H}:=\{ z\in \mathbb{C}| \textup{Im}(z)>0 \}$, and define $\mathbb{C}^*:=\mathbb{C}-\{0\}$.  We recall the $q$-Pochhammer notation:
\begin{equation*}
(x)_n=(x;q)_n:=\prod_{i=0}^{n-1}(1-q^ix), \ \ (x)_{\infty}=(x;q)_{\infty}:=\prod_{i\ge 0}(1-q^ix).
\end{equation*}

We begin with famous results of Ramanujan on congruence properties of the partition function.  A partition of a positive integer $n$ is a weakly-decreasing sequence of positive integers whose sum is equal to $n$.  The partitions of $4$ read 
\begin{equation*}
(4), \ (3,1), \ (2,2), \ (2,1,1), \ (1,1,1,1).  
\end{equation*}
We denote the number of partitions of $n$ by $p(n)$, and we define $p(0):=1$.  The following congruences were discovered and proved by Ramanujan:
\begin{align*}
p(5n+4)&\equiv 0 \pmod 5,\\
p(7n+5)&\equiv 0 \pmod 7,\\
p(11n+6)&\equiv 0 \pmod{11}.
\end{align*}
Ramanujan also found generating functions for $p(5n+4)$ and $p(7n+5)$ \cite{Ram}.  The two generating functions have elegant expressions in terms of infinite products.  In particular, we have
\begin{equation}
\sum_{n=0}^{\infty}p(5n+4)q^{n}=5\frac{(q^5;q^5)_{\infty}^5}{(q;q)_{\infty}^6},\label{equation:Ram-gen1}
\end{equation}
where as
\begin{equation}
\sum_{n=0}^{\infty}p(7n+5)q^{n}=7\frac{(q^7;q^7)_{\infty}^3}{(q;q)_{\infty}^4}
+49q\frac{(q^7;q^7)_{\infty}^7}{(q;q)_{\infty}^8}.\label{equation:Ram-gen2}
\end{equation}
Chan and Mao  \cite{CM}, Mao \cite{Mao}, and Brietzke, da Silva, and Sellars \cite{BdSS} have obtained analogous results for the eighth order mock theta function $V_{0}(q)$.  

Before we state their results, we introduce some notation.  We shall use theta functions in the form
\begin{gather*}
 j(x;q):=(x)_{\infty}(q/x)_{\infty}(q)_{\infty}=\sum_{n=-\infty}^{\infty}(-1)^nq^{\binom{n}{2}}x^n,
\end{gather*}
where the equality between product and sum follows from Jacobi's triple product identity.    We let $a$ and $m$ be integers with $m$ positive, we then define
\begin{gather*}
J_{a,m}:=j(q^a;q^m), \ \ J_m:=J_{m,3m}=\prod_{i\ge 1}(1-q^{mi}), \ {\text{and }}\overline{J}_{a,m}:=j(-q^a;q^m).
\end{gather*}
We also recall Appell functions, which are building blocks for Ramanujan's classical mock theta functions.  We will define Appell functions as follows
\begin{equation}
m(x,z;q):=\frac{1}{j(z;q)}\sum_{r=-\infty}^{\infty}\frac{(-1)^rq^{\binom{r}{2}}z^r}{1-q^{r-1}xz}.
\end{equation}
From \cite{GM} and \cite[Section 5]{HM}, we have for the eighth order mock theta function $V_0(q)$ that
\begin{align*}
V_0(q)=\sum_{n=0}^{\infty}g(n)q^{n}&:=-1+2\sum_{n\ge 0}\frac{q^{n^2}(-q;q^2)_n}{(q;q^2)_{n}}=-1+2\sum_{n\ge 0}\frac{q^{2n^2}(-q^2;q^4)_n}{(q;q^2)_{2n+1}}
\\
&\ =-q^{-1}m(1,q;q^8)-q^{-1}m(1,q^3;q^8).
\end{align*}

The results of  Chan and Mao  \cite{CM}, Mao \cite{Mao}, and Brietzke, da Silva, and Sellars \cite{BdSS} can be stated compactly as follows.  We point out that there are other results in \cite{BdSS, CM, Mao} as well.
\begin{theorem}\label{theorem:main} \cite{BdSS, CM, Mao} We have
{\allowdisplaybreaks \begin{align}
\sum_{n=0}^{\infty}g(8n)q^{n}
&=\frac{1}{J_{1}^3}\overline{J}_{1,4} \left ( \overline{J}_{1,2}^2\overline{J}_{2,4}
 +q\overline{J}_{0,2}^2\overline{J}_{0,4}\right ),\label{equation:V0-0}\\
\sum_{n=0}^{\infty}g(8n+1)q^{n}
&=\frac{2}{J_{1}^3}\overline{J}_{1,2}^3\overline{J}_{2,8} , 
%=2\frac{(q^2;q^2)_{\infty}^{14}}{(q;q)_{\infty}^9(q^4;q^4)_{\infty}^4}
\label{equation:V0-1}\\
\sum_{n=0}^{\infty}g(8n+2)q^{n}
&=\frac{4}{J_{1}^3}\overline{J}_{1,4}\overline{J}_{1,2}\overline{J}_{2,8}\overline{J}_{2,4},
%=4\frac{(q^2;q^2)_{\infty}^4(q^4;q^4)_{\infty}^5}{(q;q)_{\infty}^6(q^8;q^8)_{\infty}^2}
 \label{equation:V0-2}\\
\sum_{n=0}^{\infty}g(8n+3)q^{n}
&=\frac{4}{J_{1}^3}
\overline{J}_{1,2}^2\overline{J}_{2,8}^2,  
%=4\frac{(q^2;q^2)_{\infty}^8}{(q;q)_{\infty}^7}
 \label{equation:V0-3}\\
\sum_{n=0}^{\infty}g(8n+4)q^{n}
&=\frac{1}{J_{1}^3}\overline{J}_{1,4}\left ( \overline{J}_{1,2}^2\overline{J}_{0,4}
+ \overline{J}_{0,2}^2 \overline{J}_{2,4}\right ), \label{equation:V0-4}\\
\sum_{n=0}^{\infty}g(8n+5)q^{n}
&=\frac{8}{J_{1}^3}
\overline{J}_{2,4}\overline{J}_{1,2}\overline{J}_{2,8}\overline{J}_{4,16},
%=8\frac{(q^2;q^2)_{\infty}(q^4;q^4)_{\infty}^5}{(q^5;q^5)_{\infty}^5}
 \label{equation:V0-5} \\
\sum_{n=0}^{\infty}g(8n+6)q^{n}
&=\frac{8}{J_{1}^3}\overline{J}_{1,4}\overline{J}_{1,2}\overline{J}_{2,8}\overline{J}_{4,16}, \label{equation:V0-6}\\
%=8\frac{(q^2;q^2)_{\infty}^6(q^8;q^8)_{\infty}^2}{(q;q)_{\infty}^6(q^4;q^4)_{\infty}}
%
\sum_{n=0}^{\infty}g(8n+7)q^{n}
&=-2q^{-1} m(1,-1;q)\\
&\qquad +\frac{1}{2}\frac{q^{-1}}{J_{1}^3}
\left ( \overline{J}_{1,2}^2\overline{J}_{2,4}^2 
+q\overline{J}_{0,4} \left ( 2\overline{J}_{2,4} \overline{J}_{0,2}^2
+\overline{J}_{1,2}^2\overline{J}_{0,4}
\right)\right).\notag
\end{align}}%
\end{theorem}

Proofs of the above identities can be found scattered across \cite{BdSS, CM, Mao}.  For example, identities (\ref{equation:V0-3}) and (\ref{equation:V0-6}) are proved in \cite{CM}, identities (\ref{equation:V0-2}) and (\ref{equation:V0-5}) are proved in \cite{Mao}, and others are found in \cite{BdSS}.  In \cite{BdSS, CM} we find classical proofs, and in \cite{Mao} we find a proof that employs mock modularity.   In this note, we present a short, elementary proof using Appell functions and their properties.

Instead of asking when the generating function for an arithmetic progression of Fourier coefficients is a simple quotient of theta functions (\ref{equation:Ram-gen1}) or a handful of simple quotients (\ref{equation:Ram-gen2}), one could ask when the generating function is modular.  We use Freeman Dyson's ranks and celebrated work of Bringmann--Ono and Bringmann--Ono--Rhoades as motivation.

Dyson conjectured a combinatorial explanation of Ramanujan's congruences for the partition function using a statistic, which he called the rank.  One defines the rank of a partition to be the largest part of the partition minus the number of parts.   The ranks of the five partitions of $4$ read $3,1,0,-1,-3$, thus giving a distribution of the partitions of $4$ into five classes of equal size depending on the value of the rank modulus $5$.  The rank explains the first two of Ramanujan's congruences, but not the third \cite{ASD}.

For a more detailed look at Dyson's ranks, we define
\begin{equation*}
N(a,M;n):=\textup{number of partitions of $n$ with rank $\equiv a \pmod M$} 
\end{equation*}
which has the property that $N(a,M;n)=N(M-a,M;n)$.  Among Dyson's conjectures one finds
\begin{gather*}
N(0,5;5m+4)=N(1,5;5m+4)=N(2,5;5m+4)\\
N(0,7;7m+5)=N(1,7;7m+5)=N(2,7;7m+5)=N(3,7;7n+5),
\end{gather*}
which together with the symmetry property give the first two of Ramanujan's congruences.   For an even more in detailed look, we define the rank-difference
\begin{equation*}
R(a,b,M,c,m;q):=\sum_{n=0}^{\infty}\left ( N(a,M;mn+c)-N(b,M;mn+c)\right )q^n, 
\end{equation*}
where $a,b,c,m,M$ are integers with $0\le a,b <M$ and $0\le c < m$.  Dyson's conjectures can also be written in terms of rank-differences that are equal to zero.
\begin{gather*}
R(0,1,5,4,5;q)=R(0,2,5,4,5;q)=0,\\
R(0,1,7,5,7;q)=R(0,2,7,5,7;q)=R(0,3,7,5,7;q).
\end{gather*}

Atkin and Swinnerton-Dyer not only proved Dyson's conjectures \cite{ASD}, but they also determined rank-differences that are equal to mock theta and theta functions.  For example, they established identities such as
\begin{gather*}
R(0,2,5,1,5;q)=\frac{(q^5;q^5)_{\infty}^2}{J_{1,5}},\\
R(1,2,5,2,5;q)=\frac{(q^5;q^5)_{\infty}^2}{J_{2,5}}.
\end{gather*}

In celebrated work, Bringmann and Ono \cite{BrO2} and later Bringman, Ono, and Rhoades \cite{BrOR} found generalizations of Atkin and Swinnerton-Dyer's results on rank-differences.  In short, for a prime $t\ge5$, $0\le r_1, r_2<t$ and $0\le d<t$, they found conditions such that
\begin{equation*}
\sum_{n=0}^{\infty}\left (N(r_1,t;tn+d)-N(r_2,t;tn+d)\right )q^{24(tn+d)-1}
\end{equation*}
is a weight $1/2$ weakly holomorphic modular form on the congruence subgroup $\Gamma_{1}(576t^6)$ \cite[Theorem 1.1]{BrOR}.  One also finds an approach to this result using new Appell function properties \cite[Section 3]{HM} in \cite{HM2}.

Using the approach of \cite{HM2}, we can generalize Theorem \ref{theorem:main}  to results in the spirit of those found in \cite{BrO2, BrOR}, but we will keep our discussion within the setting of $q$-series.
\begin{theorem}\label{theorem:general} For $p\ge3$ an odd integer, we have
\begin{align*}
V_{0}(q)
&=-2\sum_{r=0}^{p-1}(-1)^{r}q^{-(2r+1)^2}m\Big (q^{8p\left (\tfrac{p-1}{2}-r\right )},-1;q^{8p^2}\Big )\\
&\qquad +\frac{J_{8p}^3}{J_{1,8}\overline{J}_{0,8p^2}}\sum_{r=0}^{p-1}
\frac{(-1)^rq^{4r(r-p)-1}q^{r-(p-1)/2}j(-q^{8r+1};q^{8p})j(-q^{p(8r+1)};q^{8p^2})}{\overline{J}_{0,8p}j(q^{8r+1};q^{8p})}\\
&\qquad +\frac{J_{8p}^3}{J_{3,8}\overline{J}_{0,8p^2}}\sum_{r=0}^{p-1}
\frac{(-1)^rq^{4r(r-p)-1}q^{3r-3(p-1)/2}j(-q^{8r+3};q^{8p})j(-q^{p(8r+3)};q^{8p^2})}{\overline{J}_{0,8p}j(q^{8r+3};q^{8p})}.
\end{align*}
\end{theorem}

To relate this to Theorem \ref{theorem:main}, we note that if for all $r\in \{ 0, 1, \dots p-1\}$, we have that
\begin{equation*}
\alpha\not \equiv -(2r+1)^2 \pmod{8p},
\end{equation*}
then
\begin{equation*}
\sum_{n=0}^{\infty}g(8pn+\alpha)q^{8pn+\alpha}
\end{equation*}
only receives contributions from the two sums consisting of theta functions.  If we consider $p=5$, then for  $\alpha \not \equiv   15, 31, 39 \pmod{40}$,  it turns out that 
\begin{equation*}
\sum_{n=0}^{\infty}g(40n+\alpha)q^{40n+\alpha}
\end{equation*}
only receives contributions from the theta functions.

\section{Preliminaries}\label{section:prelim}
We will frequently use the following identities without mention.
\begin{subequations}
{\allowdisplaybreaks \begin{gather}
\overline{J}_{0,1}=2\overline{J}_{1,4}=\frac{2J_2^2}{J_1},  \ \overline{J}_{1,2}=\frac{J_2^5}{J_1^2J_4^2},   \ J_{1,2}=\frac{J_1^2}{J_2},   \ \overline{J}_{1,3}=\frac{J_2J_3^2}{J_1J_6}, \notag\\
J_{1,4}=\frac{J_1J_4}{J_2},  \  J_{1,6}=\frac{J_1J_6^2}{J_2J_3},   \ \overline{J}_{1,6}=\frac{J_2^2J_3J_{12}}{J_1J_4J_6}.\notag
\end{gather}}%
\end{subequations}
Also following from the definitions are the general identities:
\begin{subequations}
{\allowdisplaybreaks \begin{gather}
j(q^n x;q)=(-1)^nq^{-\binom{n}{2}}x^{-n}j(x;q), \ \ n\in\mathbb{Z},\label{equation:1.8}\\
j(x;q)=j(q/x;q)=-xj(x^{-1};q)\label{equation:1.7},\\
j(x;q)={J_1}j(x,qx,\dots,q^{n-1}x;q^n)/{J_n^n} \ \ {\text{if $n\ge 1$,}}\label{equation:1.10}\\
j(z;q)=\sum_{k=0}^{m-1}(-1)^k q^{\binom{k}{2}}z^k
j\big ((-1)^{m+1}q^{\binom{m}{2}+mk}z^m;q^{m^2}\big ).\label{equation:jsplit}
\end{gather}}%
\end{subequations}
\noindent  We have two useful specializations of (\ref{equation:jsplit}):
\begin{gather}
j(z;q)=j(-qz^2;q^4)-zj(-q^3z^2;q^4),
\label{equation:jsplit-n2}\\
j(z;q)=j(-q^6z^4;q^{16})-zj(-q^{10}z^4;q^{16})+qz^2j(-q^{14}z^4;q^{16})
-q^3z^3j(-q^{18}z^4;q^{16})
\label{equation:jsplit-n4}.
\end{gather}

We have the three-term Weierstrass relation for theta functions \cite[(1.)]{We}, \cite{Ko}:  For generic $a,b,c,d\in \mathbb{C}^*$
\begin{align}
j(ac;q)j(a/c;q)j(bd;q)j(b/d;q)&=j(ad;q)j(a/d;q)j(bc;q)j(b/c;q)
\label{equation:Weier-3term}\\
&\qquad +b/c \cdot j(ab;q)j(a/b;q)j(cd;q)j(c/d;q).\notag
\end{align}

We will also need the following:
\begin{proposition} [Theorem $1.1$ \cite{H1}] For generic $x,y,z\in \mathbb{C}^*$ 
 \begin{gather}
j(x;q)j(y;q)=j(-xy;q^2)j(-qx^{-1}y;q^2)-xj(-qxy;q^2)j(-x^{-1}y;q^2),
\label{equation:H1Thm1.1}
\end{gather}
\end{proposition}

\begin{lemma} \label{lemma:lemma-A}We have
\begin{equation}
\overline{J}_{1,8}^2 J_{3,8}^2+\overline{J}_{3,8}^2 J_{1,8}^2
 = \overline{J}_{0,8}J_{4,8}^2\overline{J}_{2,8}.\label{equation:lemma-A1}
\end{equation}
\end{lemma}
\begin{proof}[Proof of Lemma \ref{lemma:lemma-A}] We use (\ref{equation:Weier-3term}) with $q\to q^{8}$, $a\to -q^{4}$, $b\to -q^{3}$, $c\to q^{3}$, $d\to -1$.
\end{proof}

\begin{lemma} \label{lemma:lemma-B} We have
\begin{gather}
\overline{J}_{16,32}^2+q^{8}\overline{J}_{0,32}^2
= \overline{J}_{8,16}^2,\label{equation:lemma-B1}\\
2\overline{J}_{8,16}^2\overline{J}_{0,32}\overline{J}_{16,32}
+ \overline{J}_{16,32}^2\overline{J}_{0,16}^2
+q^{8}\overline{J}_{0,32}^2\overline{J}_{0,16}^2
=2\overline{J}_{8,16}^2\overline{J}_{0,16}^2.\label{equation:lemma-B2}
\end{gather}
\end{lemma}
\begin{proof}[Proof of Lemma \ref{lemma:lemma-B}]  For (\ref{equation:lemma-B1}), we use (\ref{equation:H1Thm1.1}) with $q\to q^{16}$, $x\to -q^{8}$, $y\to -q^{8}$.  For (\ref{equation:lemma-B2}), we proceed as follows:
{\allowdisplaybreaks \begin{align*}
2\overline{J}_{8,16}^2\overline{J}_{0,32}\overline{J}_{16,32}
+ \overline{J}_{16,32}^2\overline{J}_{0,16}^2
+q^{8}\overline{J}_{0,32}^2\overline{J}_{0,16}^2
&=2\overline{J}_{8,16}^2\overline{J}_{0,32}\overline{J}_{16,32}
+ \overline{J}_{0,16}^2\overline{J}_{8,16}^2\\
&=\overline{J}_{8,16}^2\left( 2\overline{J}_{0,32}\overline{J}_{16,32}+ \overline{J}_{0,16}^2\right)\\
&=\overline{J}_{8,16}^2\left( 2\overline{J}_{0,16}\overline{J}_{16,64}
+  \overline{J}_{0,16}^2\right)\\
&=2\overline{J}_{8,16}^2\overline{J}_{0,16}^2,
\end{align*}}%
where the first equality follows from (\ref{equation:lemma-B1}) and the third and fourth equalities follow from product rearrangements.
\end{proof}

We also have the Appell function property \cite{HM, Zw2}
\begin{equation}
m(x,z_1;q)-m(x,z_0;q)=\frac{z_0(q)_{\infty}^3j(z_1/z_0;q)j(xz_0z_1;q)}{j(z_0;q)j(z_1;q)j(xz_0;q)j(xz_1;q)}. \label{equation:changing-z}
\end{equation}

\section{Proof of Theorem \ref{theorem:main}}

We begin by using (\ref{equation:changing-z}) and collecting terms
{\allowdisplaybreaks \begin{align*}
V_0(q)&=-q^{-1}m (1,q;q^8 )-q^{-1}m(1,q^3;q^8 )\\
&=-q^{-1}\left( m(1,-1;q^8)-\frac{J_{8}^3\overline{J}_{1,8}^2}{J_{1,8}^2\overline{J}_{0,8}^2}
+m(1,-1;q^8)-\frac{J_{8}^3\overline{J}_{3,8}^2}{J_{3,8}^2\overline{J}_{0,8}^2}\right)\\
&=-2q^{-1} m(1,-1;q^8)+q^{-1}\frac{J_{8}^3}{\overline{J}_{0,8}^2}
\frac{1}{J_{1,8}^2J_{3,8}^2}
\left ( \overline{J}_{1,8}^2 J_{3,8}^2 + \overline{J}_{3,8}^2 J_{1,8}^2 \right ).
\end{align*}}%
Using (\ref{equation:1.10}) we rewrite a product of theta functions in the denominator to get
{\allowdisplaybreaks \begin{align*}
V_0(q)&
=-2q^{-1} m(1,-1;q^8)+q^{-1}\frac{J_{8}^3}{\overline{J}_{0,8}^2}\frac{1}{J_{1,4}^2}\frac{J_{4}^2}{J_{8}^4}
\left ( \overline{J}_{1,8}^2 J_{3,8}^2 + \overline{J}_{3,8}^2 J_{1,8}^2 \right ).
\end{align*}}%
Using Lemma \ref{lemma:lemma-A} we obtain
{\allowdisplaybreaks \begin{align*}
V_0(q)&
=-2q^{-1} m(1,-1;q^8)+q^{-1}\frac{J_{8}^3}{\overline{J}_{0,8}^2}\frac{1}{J_{1,4}^2}\frac{J_{4}^2}{J_{8}^4}
\left ( \overline{J}_{0,8}J_{4,8}^2\overline{J}_{2,8}\right ) \\
&=-2q^{-1} m(1,-1;q^8)+q^{-1}\frac{1}{2}\frac{\overline{J}_{1,2}\overline{J}_{2,4}^2\overline{J}_{4,8}}{J_8^3}.
\end{align*}}%
We expand $\overline{J}_{1,2}$ with (\ref{equation:jsplit-n4}) and $\overline{J}_{2,4}$ and $\overline{J}_{4,8}$ with (\ref{equation:jsplit-n2}) to obtain
{\allowdisplaybreaks \begin{align*}
V_0(q)
&=-2q^{-1} m(1,-1;q^8)+q^{-1}\frac{1}{2}\frac{\overline{J}_{1,2}\overline{J}_{2,4}^2\overline{J}_{4,8}}{J_8^3}\\
&=-2q^{-1} m(1,-1;q^8)\\
&\qquad +\frac{1}{2}\frac{q^{-1}}{J_{8}^3}
\left ( \overline{J}_{16,32}+q\overline{J}_{24,32}+q^4\overline{J}_{0,32}+q^9\overline{J}_{40,32}\right ) 
\left ( \overline{J}_{8,16}+q^2\overline{J}_{16,16}\right )^2\\
&\qquad \qquad \cdot \left ( \overline{J}_{16,32}+q^4\overline{J}_{32,32}\right ).
\end{align*}}%
We simplify using Identity (\ref{equation:1.8})
{\allowdisplaybreaks \begin{align*}
V_0(q)
&=-2q^{-1} m(1,-1;q^8)\\
&\qquad +\frac{1}{2}\frac{q^{-1}}{J_{8}^3}
\left ( \overline{J}_{16,32}+2q \overline{J}_{8,32}+q^4\overline{J}_{0,32}\right ) 
\left ( \overline{J}_{8,16}+q^2\overline{J}_{0,16}\right )^2
\left ( \overline{J}_{16,32}+q^4\overline{J}_{0,32}\right ).
\end{align*}}
Then we distribute the products to get
{\allowdisplaybreaks \begin{align*}
V_0(q)
&=-2q^{-1} m(1,-1;q^8)\\
&\qquad +\frac{1}{2}\frac{q^{-1}}{J_{8}^3}
\Big ( \overline{J}_{8,16}^2\overline{J}_{16,32}^2 
+q^8\overline{J}_{0,32} \left ( 2\overline{J}_{16,32} \overline{J}_{0,16}^2
+\overline{J}_{8,16}^2\overline{J}_{0,32}
\right)\\
&\qquad \qquad 
+2q \left ( \overline{J}_{8,16}^2\overline{J}_{8,32}\overline{J}_{16,32}
 +q^8\overline{J}_{8,32}\overline{J}_{0,16}^2\overline{J}_{0,32}\right )\\
&\qquad \qquad
+2q^2\left (\overline{J}_{16,32}^2\overline{J}_{8,16}\overline{J}_{0,16}
+q^{8}\overline{J}_{0,32}\overline{J}_{8,16}\overline{J}_{0,16}\overline{J}_{0,32}\right )\\
&\qquad \qquad 
+4q^3\overline{J}_{8,32}\overline{J}_{8,16}\overline{J}_{0,16}\overline{J}_{16,32}\\
&\qquad \qquad 
+q^4\left( 2\overline{J}_{8,16}^2\overline{J}_{0,32}\overline{J}_{16,32}
+ \overline{J}_{16,32}^2\overline{J}_{0,16}^2
+q^{8}\overline{J}_{0,32}^2\overline{J}_{0,16}^2\right ) \\
& \qquad \qquad 
+2q^5 \overline{J}_{8,32}\left ( \overline{J}_{8,16}^2\overline{J}_{0,32}
+ \overline{J}_{0,16}^2 \overline{J}_{16,32}
\right) \\
&\qquad \qquad 
+4q^6 
\overline{J}_{16,32}\overline{J}_{8,16}\overline{J}_{0,16}\overline{J}_{0,32} \\
&\qquad\qquad
+4q^7\overline{J}_{8,32}\overline{J}_{8,16}\overline{J}_{0,16}\overline{J}_{0,32}
\Big),
\end{align*}}%
where we have grouped terms according to the values of $q$-exponents modulo $8$.  From this we have
{\allowdisplaybreaks \begin{align*}
\sum_{n=0}^{\infty}g(8n)q^{8n}
&=\frac{1}{J_{8}^3}\overline{J}_{8,32} \left ( \overline{J}_{8,16}^2\overline{J}_{16,32}
 +q^8\overline{J}_{0,16}^2\overline{J}_{0,32}\right ),\\
\sum_{n=0}^{\infty}g(8n+1)q^{8n+1}
&=\frac{q}{J_{8}^3}\overline{J}_{8,16}\overline{J}_{0,16}\left (\overline{J}_{16,32}^2
+q^{8}\overline{J}_{0,32}^2\right ),\\
\sum_{n=0}^{\infty}g(8n+2)q^{8n+2}
&=\frac{2q^2}{J_{8}^3}\overline{J}_{8,32}\overline{J}_{8,16}\overline{J}_{0,16}\overline{J}_{16,32},\\
\sum_{n=0}^{\infty}g(8n+3)q^{8n+3}
&=\frac{1}{2}\frac{q^3}{J_{8}^3}\left( 2\overline{J}_{8,16}^2\overline{J}_{0,32}\overline{J}_{16,32}
+ \overline{J}_{16,32}^2\overline{J}_{0,16}^2
+q^{8}\overline{J}_{0,32}^2\overline{J}_{0,16}^2\right ),\\
\sum_{n=0}^{\infty}g(8n+4)q^{8n+4}
&=\frac{q^4}{J_{8}^3}\overline{J}_{8,32}\left ( \overline{J}_{8,16}^2\overline{J}_{0,32}
+ \overline{J}_{0,16}^2 \overline{J}_{16,32}\right ),\\
\sum_{n=0}^{\infty}g(8n+5)q^{8n+5}
&=\frac{2q^5}{J_{8}^3}
\overline{J}_{16,32}\overline{J}_{8,16}\overline{J}_{0,16}\overline{J}_{0,32},\\
\sum_{n=0}^{\infty}g(8n+6)q^{8n+6}
&=\frac{2q^6}{J_{8}^3}\overline{J}_{8,32}\overline{J}_{8,16}\overline{J}_{0,16}\overline{J}_{0,32},\\
\sum_{n=0}^{\infty}g(8n+7)q^{8n+7}
&=-2q^{-1} m(1,-1;q^8)\\
&\qquad +\frac{1}{2}\frac{q^{-1}}{J_{8}^3}
\left ( \overline{J}_{8,16}^2\overline{J}_{16,32}^2 
+q^8\overline{J}_{0,32} \left ( 2\overline{J}_{16,32} \overline{J}_{0,16}^2
+\overline{J}_{8,16}^2\overline{J}_{0,32}
\right)\right).
\end{align*}}%
To obtain the final forms found in Theorem \ref{theorem:main}, we use Identity (\ref{equation:lemma-B1}) for the second equality and Identity (\ref{equation:lemma-B2}) for the fourth equality.

\section{Proof of Theorem \ref{theorem:general}}

Let $p$ be a positive odd integer.   For generic $x,z,z'\in \mathbb{C}^*$  \cite[Corollary $3.6$]{HM} 
{\allowdisplaybreaks \begin{align}
m(x,&z;q)=\sum_{r=0}^{p-1}q^{-\binom{r+1}{2}}(-x)^rm\Big (q^{\binom{p}{2}-pr}x^p,z';q^{p^2}\Big )\\
&+\frac{z'J_p^3}{j(xz;q)j(z';q^{p^2})}\sum_{r=0}^{p-1}
\frac{q^{r(r-p)/2}(-x)^rz^{r-(p-1)/2}j(q^rx^pzz';q^p)j(q^{pr}z^p/z';q^{p^2})}{j(x^pz';q^p)j(q^rz;q^p)}.\notag
\end{align}
We further specialize the above theorem to
{\allowdisplaybreaks \begin{align*}
m&(1,z;q^8)\\
&=\sum_{r=0}^{p-1}q^{-4r(r+1)}(-1)^rm\Big (q^{4p(p-1)-8pr},-1;q^{8p^2}\Big )\\
&\qquad -\frac{J_{8p}^3}{j(z;q^{8})\overline{J}_{0,8p^2}}\sum_{r=0}^{p-1}
\frac{q^{4r(r-p)}(-1)^rz^{r-(p-1)/2}j(-q^{8r}z;q^{8p})j(-q^{8pr}z^p;q^{8p^2})}{\overline{J}_{0,8p}j(q^{8r}z;q^{8p})}.
\end{align*}
We recall that \cite[Section 5]{HM}
\begin{align*}
V_{0}(q)&=-q^{-1}m(1,q;q^8)-q^{-1}m(1,q^3;q^8),
\end{align*}
so it suffices to compute the two specializations for $z=q$ and $q^3$ and then add them in order to finish the proof.

\section*{Acknowledgements}
The work is supported by the Ministry of Science and Higher Education of the Russian
Federation (agreement no. 075-15-2022-287).

\end{document}